\documentclass[12pt,a4paper]{amsart}
\usepackage{amsmath,amssymb,color,setspace}
\usepackage{hyperref}
\usepackage[utf8,utf8x]{inputenc}
\usepackage{fullpage}

\usepackage{array}
\usepackage{epsfig}
\usepackage{xspace}
\usepackage{color}

\usepackage{enumitem}

\newtheorem{propo}{Proposition}[section]

\newtheorem{theor}[propo]{Theorem}
\newtheorem{lemma}[propo]{Lemma}

\newtheorem{question}[propo]{Question}
\newtheorem*{theor*}{Theorem}

\theoremstyle{definition}
\newtheorem{defin}[propo]{Definition}
\newtheorem{examp}[propo]{Example}

\theoremstyle{remark}
\newtheorem{remar}[propo]{Remark}

\newcommand{\NN }{\mathbb{N}}

\newcommand{\FF }{\mathbb{F}}
\newcommand{\QQ }{\mathbb{Q}}
\newcommand{\ZZ }{\mathbb{Z}}
\newcommand{\PP }{\mathbb{P}}

\newcommand{\DR }{R}
\newcommand{\KR }{K}
\DeclareMathOperator{\Gr}{Gr_{\it \KR}}
\DeclareMathOperator{\Quot}{Quot}
\newcommand{\AC }{\mathcal{A}}
\DeclareMathOperator{\SL}{SL}

\DeclareMathOperator{\sgn}{sgn  }

\newcommand{\pp}{q}

\DeclareMathOperator{\rank}{rank}
\DeclareMathOperator{\rk}{rk}
\newcommand{\MC }{\mathcal{M}}
\newcommand{\yy }{\omega}

\definecolor{darkgreen}{rgb}{0.0,0.1,0.6}
\newcommand{\df}[1]{{\bf\color{darkgreen} #1}}

\title[Grassmannians over rings and subpolygons]
{Grassmannians over rings and subpolygons}

\author{Michael~Cuntz}
\address{Michael Cuntz, Leibniz Universit\"at Hannover,
Institut f\"ur Algebra, Zah\-lentheorie und Diskrete Mathematik,
Fakult\"at f\"ur Mathematik und Physik,
Wel\-fengarten 1,
D-30167 Hannover, Germany}
\email{cuntz@math.uni-hannover.de}

\begin{document}


\begin{abstract}
We investigate special points on the Grassmannian which correspond to friezes with coefficients in the case of rank two. Using representations of arithmetic matroids we obtain a theorem on subpolygons of specializations of the coordinate ring. As a special case we recover the characterization of subpolygons in classic frieze patterns. Moreover, we observe that specializing clusters of the coordinate ring of the Grassmannian to units yields representations that may be interpreted as arrangements of hyperplanes with notable properties. In particular, we get an interpretation of certain Weyl groups and groupoids as generalized frieze patterns.
\end{abstract}

\maketitle

\section{Introduction}

The coordinate ring of a Grassmannian $\Gr(k,n)$ is known to have a cluster structure \cite{MR2205721}.
When $k=2$, the clusters correspond to triangulations of an $n$-gon by non-intersecting diagonals.
Under this correspondence, Pl\"ucker coordinates $p_{i,j}$ are edges and diagonals between vertices $i,j$; they satisfy the so-called Ptolemy relations. Specializing all variables of a cluster to some values in a commutative ring $\DR$ we obtain a map
from the set of diagonals and edges of the $n$-gon to $\DR$.
Coxeter mentioned such a map in \cite{MR1119304},
in \cite{MR4080476} we call it a \df{frieze with coefficients}; it is a frieze pattern in the classical sense when the frozen variables (the edges of the $n$-gon) are mapped to $1$ \cite{MR2250855}.

By the Laurent phenomenon, specializing all variables of a cluster to $1$ gives a frieze with values in $\ZZ_{>0}$, it is then called a \df{Conway-Coxeter frieze}. Restricting the frieze to a subpolygon yields a frieze with values in the same ring. However, the frozen variables are not necessarily specialized to $1$ since the edges of the subpolygon were originally diagonals. In \cite{MR4080476} we raise the question to describe friezes with integer values that are obtained from subpolygons in Conway-Coxeter frieze patterns.
We give an answer to the case of triangles which is somewhat surprising at first sight:

\begin{theor*}[{\cite[Thm.\ 5.12]{MR4080476}}]
\label{thm:triangle}
A triple $(a,b,c)\in\NN^3$ appears as labels of a triangle
in some Conway-Coxeter frieze if and only if the following two conditions are
satisfied:
\begin{enumerate}
\item $\gcd(a,b)=\gcd(b,c)=\gcd(a,c)$,
\item $\nu_2(a)=\nu_2(b)=\nu_2(c)=0$ \quad or\quad $|\{\nu_2(a),\nu_2(b),\nu_2(c)|>1$
\end{enumerate}
where $\nu_2(\cdot )$ denotes the 2-valuation.
\end{theor*}

There is this strange appearance of $2$-valuations which was elucidated in our more general result:

\begin{theor*}[{\cite[Thm.\ 3.2]{MR4311087}}]
Let $\mathcal{C}$ be a frieze with coefficients 
on an $n$-gon over positive integers. Then $\mathcal{C}$ appears as a subpolygon
of some Conway-Coxeter frieze if and only if the following conditions are
satisfied:
\begin{enumerate}
\item \label{cond0:gcd}
For any triangle $(a,b,c)$ in $\mathcal{C}$ we have
$\gcd(a,b)=\gcd(b,c)=\gcd(a,c)$.
\item \label{cond0:p+1}
Let $p<n$ be a prime number. Then for each $(p+1)$-subpolygon $\mathcal{D}$
of $\mathcal{C}$ the labels of edges and diagonals in $\mathcal{D}$ are either
all not divisible by $p$ or they do not all have the same $p$-valuation.
\end{enumerate}
\end{theor*}

This explains why $2$-valuations play a role in the case of triangles.
However, the conditions still look mysterious.
In this paper, we generalize the question to $\SL_k$-friezes, viewed as in \cite{MR3264259} as points on arbitrary Grassmannians\footnote{arbitrary $k$ and almost arbitrary ring of coefficients}.
We call such a generalized frieze a \df{specialization} of the coordinate ring $\AC_{k,n}$.
It turns out that the $p$-valuations yet have a better explanation:
the $p+1$-subpolygons of the previous theorem have $p+1$ vertices
because this is the number of elements in the projective line over $\FF_p$.

\begin{theor*}[Thm.\ \ref{thm:main}]
Let $\DR$ be a principal ideal domain and
$\psi$ be a specialization of $\AC_{k,n}$ such that
\[ \psi(t_{i_1,\ldots,i_k})\in \DR\setminus\{0\} \]
for all $1\le i_1<\ldots<i_k \le n$.
Let $\varepsilon\in R$ be such that
$$(\varepsilon)=(\psi(t_{i_1,\ldots,i_k}) \mid 1\le i_1 < \ldots < i_k \le n).$$
We write $(\varepsilon)=(\pp_1^{d_1})\cap\cdots\cap(\pp_e^{d_e})$ for primes $\pp_1,\ldots,\pp_e$.

Then $\psi$ has a representation
$X=(x_1 \cdots x_n)\in \DR^{k\times n}$
such that\footnote{$|x|$ denotes the greatest common divisor of the coordinates of $x$.} $|x_i|=1$ for all $1\le i \le n$ if and only if the following conditions are satisfied:
\begin{enumerate}
\item For all $1\le i \le n$, the ideals
\[ (\psi(t_{i,i_2,\ldots,i_k}) \mid 1\le i_2 < \ldots < i_k \le n)
= (\varepsilon) \]
coincide.
\item For each $\ell=1,\ldots,e$ and $S\subseteq \{1,\ldots,n\}$,
\[ |S|=|\PP (R/(\pp_\ell))^k| \quad \Longrightarrow \quad
\exists i,j \in S \:\:
\forall i_1,\ldots,i_{k-2} \:\::\:\: \nu_{\pp_\ell}(\psi(t_{i_1,\ldots,i_{k-2},i,j}) )\ne \nu_{\pp_\ell}(\varepsilon). \]
\end{enumerate}
\end{theor*}

For the formulation and proof of the theorem we need some completely different notions and ideas (see Section \ref{sec:repspe} for details).
One can view a specialization of $\AC_{k,n}$ as a certain kind of \df{matroid}:
a matroid can be defined by prescribing subsets of $\{1,\ldots,n\}$ that are called linearly dependent.
If the matroid is the matroid of linear dependencies of some set of vectors $x_1,\ldots,x_n\in \DR^k$, then this information can also be encoded by the $k$-subsets that have determinant $0$.
A specialization in our sense contains some more information, it prescribes the values of the Pl\"ucker coordinates and hence the values of all determinants. If a matroid is given by a set of vectors, then it is called representable.
We will call a set of vectors that represents a specialization of $\AC_{k,n}$ a \df{representation}.

Classically, the axioms of matroids model linear dependence of vectors in a vector space.
Since for frieze patterns over the integers we need linear dependence over $\ZZ$, we are working
with a structure similar to an arithmetic matroid instead \cite{MR2989987}, \cite{MR4186616}.
Conditions on the greatest common divisors like (1) also appear in the study of arithmetic matroids.

The proof of Theorem \ref{thm:main} uses representations and coordinates and is thus much easier to comprehend than our proofs in \cite{MR4080476} and \cite{MR4311087}: we argue with coordinate vectors instead of Pl\"ucker coordinates.

In the special case $R=\ZZ$ one can additionally require that $\psi(t_{i_1,\ldots,i_k})>0$ for all $i_1<\ldots<i_k$. Such a specialization could then possibly come from an $\SL_k$-frieze pattern.
An $\SL_k$-frieze pattern is a
specialization with variables specialized to $\NN$, frozen variables specialized to $1$, and
which has a representation $x_1,\ldots,x_n$
such that the $x_i$ have coprime coordinates, hence in particular it satisfies the condition of Theorem \ref{thm:main}.
Together with this theorem, the following second result characterizes subpolygons in $\SL_k$-frieze patterns (see Section \ref{sec:slk}):

\begin{theor*}[Thm.\ \ref{thm:slk}]
Let $X=(x_1 \cdots x_n)\in \ZZ^{k\times n}$ be a representation of a
specialization of $\AC_{k,n}$
with $|x_i|=1$ for all $i=1,\ldots,n$.
Assume that $|x_{i_1}\ldots x_{i_k}|> 0$ for all $1\le i_1<\ldots <i_k\le n$.
Then the specialization
may be extended to an $\SL_k$-frieze pattern,
i.e.\ a specialization $\psi$ with positive integer values on the cluster variables and
in which all frozen variables are specialized to $1$.
\end{theor*}

The set of all clusters together with mutations gives a structure which has many features in common with a Weyl groupoid, see \cite{MR2755086} for the case of rank two. But a connection in higher rank was nebulous so far.
In addition to the characterization of subpolygons in $\SL_k$-friezes, in Section \ref{sec:cry} we also observe a relation between the cluster structure on the coordinate ring of a Grassmannian and the structure of a Weyl groupoid:
We consider specializations of the variables of a cluster to units.
It turns out that this produces arrangements of hyperplanes with special properties.
We exhibit some examples; for instance, we recover several crystallographic arrangements in rank three, these correspond to Weyl groupoids \cite{MR2820159}.
The following question remains:

\begin{question}
Which crystallographic arrangements are defined by representations of clusters in which we specialize all variables to $\pm 1$?
\end{question}

Since for fixed $k,n\in\NN$ the number of clusters is often infinite, on the other hand the number of matroids of rank $k$ on $1,\ldots,n$ is finite, it is conceivable that any matroid which has a representation is determined by the choice of units (in some ring of integers) for the variables of a fixed cluster.

\medskip
\noindent{\bf Acknowledgement:}
{I would like to thank L.~Moci and R.~Pagaria for very helpful discussions on arithmetic matroids and S.~Morier-Genoud for comments on a previous version and further references.}

\section{Grassmannian and Pl\"ucker coordinates}

Let $\DR$ be an integral domain and $\KR:=\Quot(\DR)$ its field of fractions.
Let $n\in\NN$ and $U\le \KR^n$ be
a point on the \df{Grassmannian} $\Gr(k,n)$,
i.e.\ a subspace of $\KR^n$ of dimension $k$.
Let
\[ v_1:=\begin{pmatrix} x_{1,1} \\ \vdots \\ x_{1,n} \end{pmatrix},
\ldots, v_k:=\begin{pmatrix} x_{k,1} \\ \vdots \\ x_{k,n} \end{pmatrix} \]
be a basis of $U$ and $X=(x_{i,j})_{1\le i\le n, 1\le j\le k}$
be the matrix with the coordinates of $v_1,\ldots,v_k$ in its \emph{rows}.
We write $X_{\{i_1,\ldots,i_k\}}$ for the matrix consisting of the columns $i_1,\ldots,i_k$ of $X$ (where the labels are not required to be different or in increasing ordering), and set
\[ p_{i_1,\ldots,i_k} := \det (X_{\{i_1,\ldots,i_k\}}). \]
These elements of $\KR$ are called the \df{Pl\"ucker coordinates} of $U$. Indeed,
consider the element $v_1\wedge \ldots\wedge v_k \in \bigwedge^k \KR^n$ of the $k$-th exterior power of $\KR^n$. Then
\[ v_1\wedge \ldots\wedge v_k
= \sum_{i_1,\ldots,i_k=1}^n x_{1,i_1} e_{i_1} \wedge \ldots \wedge x_{1,i_1} e_{i_k}
= \sum_{i_1\le \ldots\le i_k=1}^n p_{i_1,\ldots,i_k} e_{i_1} \wedge \ldots \wedge e_{i_k} \]
if $e_1,\ldots,e_n$ is the standard basis of $\KR^n$. It is easy to see that a base change on $U$ yields the same
Pl\"ucker coordinates.

Now view the coordinates in $X$ as variables in the polynomial algebra
\[ P:=\KR[x_{i,j} \mid 1\le i \le k, 1\le j \le n] \]
and consider the polynomial algebra
\[ T:=\KR[t_{i_1,\ldots,i_k} \mid 1\le i_1<\ldots<i_k\le n] = S({\bigwedge}^k \KR^n). \]
Let $p$ be the homomorphism
\[ p : T \rightarrow P, \quad t_{i_1,\ldots,i_k} \mapsto \det(X_{\{i_1,\ldots,i_k\}}) = p_{i_1,\ldots,i_k}. \]
We see that $\AC_{k,n}:=T/\ker (p)$ is the coordinate ring of $\Gr(k,n)$.

When $i_1,\ldots,i_k$ are not strictly increasing, we define
$p_{i_1,\ldots,i_k}:=\sgn(\sigma) p_{i_{\sigma(1)},\ldots,i_{\sigma(k)}}$
for a permutation $\sigma\in S_k$ sorting the labels.
The following lemma is an easy exercise.
\begin{lemma}
The Pl\"ucker coordinates satisfy the following \df{Pl\"ucker relations}:
\begin{equation}\label{eq:pluecker_relation}
\sum_{r=0}^k (-1)^r p_{i_1,\ldots,i_{k-1},j_r} p_{j_0,\ldots,\hat j_r,\ldots,j_k} = 0
\end{equation}
for all $i_1,\ldots,i_{k-1}$ and $j_0,\ldots,j_k$.
\end{lemma}

In this paper, all indices $i$ with $i>n$ or $i<0$ which should be within a range $\{1,\ldots,n\}$ have to be reduced modulo $n$ into this range.
Moreover, we write $[i_1,\ldots,i_k]$ for the ordered sequence with elements $\{i_1,\ldots,i_k\}$ in ascending order.

\section{Representations and specializations}\label{sec:repspe}

\begin{defin}
Let $\AC_{k,n}$ be the coordinate ring of $\Gr(k,n)$ as defined above.
A homomorphism
\[ \psi : \AC_{k,n} \rightarrow \KR \]
is called a \df{specialization} of $\AC_{k,n}$.

A \df{representation} of $\psi$ is a matrix $X\in \KR^{k\times n}$ such that
\[ \psi(t_{i_1,\ldots,i_k}) = \det(X_{\{i_1,\ldots,i_k\}}) \]
for all $1\le i_1<\ldots<i_k \le n$.
\end{defin}

\begin{remar}
Note that we do not require that a specialization takes the value $1$ on frozen variables (see Section \ref{sec:clu}).
The case when all frozen variables are specialized to $1$ produces an $\SL_k$-frieze pattern as a representation of the specialization, see Example \ref{ex:slk} or \cite{MR4226982}.
\end{remar}

\begin{defin}
An \df{$\SL_k$-frieze pattern} is a specialization $\psi$ to $\QQ$ with
\[ \psi(t_{i_1,\ldots,i_k})\in \ZZ_{>0} \quad\text{and}
\quad \psi(t_{[i,\ldots,i+k-1]})=1 \]
for all $1\le i_1 <\cdots<i_k\le n$, $i=1,\ldots,n$.
\end{defin}

\begin{remar}
Choose a subset $\MC\subseteq \{(i_1,\ldots,i_k) \mid 1\le i_1<\ldots<i_k\le n \}$
and let $\Sigma$ be the set of specializations with
$\psi(t_{i_1,\ldots,i_k})=0$ for $(i_1,\ldots,i_k)\in \MC$.
If $\Sigma$ is not empty, then $\MC$ is a \df{matroid}.
The set of representations of all $\psi\in \Sigma$ is the moduli space of representations of $\MC$ as a matroid.
\end{remar}

The study of integral specializations (or specializations with values in a principal ideal domain) involves the following important volume function.

\begin{defin}\label{volume}
Let $R$ be a principal ideal domain and $A\in R^{k\times r}$, $r\le k$.
Then we denote by $|A|$ a generator of the ideal generated by all $r\times r$ minors of $A$.
By abuse of notation, we will write $|A|=a$ if the ideals $(|A|)=(a)$ are equal.
\\
In particular, if $R=\ZZ$, then we can choose the non-negative integer
\[ |A|=|\ZZ^k/(A\cdot \ZZ^r)|. \]
Hence $|A|$ is the product of the elementary divisors of $A$ and is equal to the \df{volume} of the parallelotope spaned by the columns of $A$.
\\
For $\beta_1,\dots,\beta_m\in R^k$, we will write
$|\beta_1 \cdots \beta_m|$ for the volume of the matrix with columns $\beta_1,\ldots,\beta_m$.
If $m=1$ and $\beta \in \ZZ^k\setminus \{0\}$, then $|\beta|$ is the greatest common divisor of the coordinates of $\beta$. We will write $|\beta|=1$ if $|\beta|$ is a unit.
\\
If $m=k$ and $\beta _1,\dots,\beta _k\in \ZZ^k$,
then we will write $|\beta _1\cdots\beta _k|$ for the
determinant of the matrix with columns $\beta _1,\dots,\beta _k$;
here we keep track of the signs.
\end{defin}

\begin{remar}
An \df{arithmetic matroid} consists of a set $E=\{1,\ldots,n\}$ together with
two maps $\rk$ and $m$ on subsets of $E$ such that $\rk$ may be viewed as the rank function of a matroid,
and, roughly speaking, $m$ `is' the volume map defined above.
Of course, these maps have to satisfy certain axioms, see \cite[Def.\ 2.3]{MR4186616} for details.
Arithmetic matroids are used to investigate toric arrangements \cite{MR2989987}.
Compared to our notion of a specialization, there are several crucial differences:
Representations of arithmetic matroids are in general chosen over the ring $R=\ZZ$ because the map $m$ takes values in $\ZZ_{>0}$.
Moreover, our specializations are much more restrictive since we include for instance informations on the signs of volumes by prescribing the exact values of all determinants.
In contrast to our specializations, there are more representations for an arithmetic matroid in general
(see Theorem \ref{thm:realization} below).
\end{remar}

The following result is almost trivial in the case when the frozen variables are specialized to $1$ (Example \ref{ex:slk}). The situation is more subtle in the general case:

\begin{theor}\label{thm:realization}
Let $\DR$ be a principal ideal domain, $K$ its field of fractions, and
$\psi$ be a specialization of $\AC_{k,n}$ such that
\[ \psi(t_{i_1,\ldots,i_k})\in \DR\setminus\{0\} \]
for all $1\le i_1<\ldots<i_k \le n$.
Then the map $\psi$ has a representation $X\in \DR^{k\times n}$.
\end{theor}
\begin{proof}
Denote $p_{i_1,\ldots,i_k}:=\psi(t_{i_1,\ldots,i_k})$.
We proceed by induction over $k$.
If $k=1$ then we may choose the representation $X=(p_1,\ldots,p_n)$.
Now assume that $k>1$.
Let $\delta$ be a greatest common divisor of $p_{1,i_2,\ldots,i_k}$ for all $1< i_2<\ldots<i_k \le n$.
Then there exist $\lambda_{i_2,\ldots,i_k}\in \DR$ such that
\[ \delta = \sum_{1< i_2<\ldots<i_k \le n} \lambda_{i_2,\ldots,i_k} p_{1,i_2,\ldots,i_k}.  \]
Using the Pl\"ucker relations \eqref{eq:pluecker_relation} we get:
\begin{eqnarray*}
\delta \cdot p_{j_1,\ldots,j_k} &=&
\sum_{1<i_2<\ldots<i_k} p_{j_1,\ldots,j_k} \lambda_{i_2,\ldots,i_k} p_{1,i_2,\ldots,i_k} \\
&=& \sum_{1<i_2<\ldots<i_k} \sum_{\nu=1}^k (-1)^{\nu-1}
p_{1,j_1,\ldots,\hat j_\nu,\ldots,j_k} \lambda_{i_2,\ldots,i_k} p_{j_\nu,i_2,\ldots,i_k} \\
&=& \sum_{\nu=1}^k (-1)^{\nu-1}
\left( \sum_{1<i_2<\ldots<i_k} \lambda_{i_2,\ldots,i_k} p_{j_\nu,i_2,\ldots,i_k} \right)
p_{1,j_1,\ldots,\hat j_\nu,\ldots,j_k}.
\end{eqnarray*}
With
\[ a_j := \sum_{1<i_2<\ldots<i_k} \lambda_{i_2,\ldots,i_k} p_{j,i_2,\ldots,i_k} \]
we obtain
\begin{equation}\label{eq:ai}
p_{j_1,\ldots,j_k} = \sum_{\nu=1}^k (-1)^{\nu-1}
a_{j_\nu} \frac{p_{1,j_1,\ldots,\widehat{j_\nu},\ldots,j_k}}{\delta}.
\end{equation}
Note that
\begin{equation}\label{eq:Xt}
\AC_{k-1,n-1} \rightarrow K, \quad
t_{i_1,\ldots,i_{k-1}} \mapsto \frac{p_{1,i_1+1,\ldots,i_{k-1}+1}}{\delta}\in \DR\setminus\{0\}
\end{equation}
is a specialization of $\AC_{k-1,n-1}$ with Pl\"ucker coordinates in $\DR\setminus\{0\}$.
Hence by induction this has a representation $\tilde X \in \DR^{(k-1)\times (n-1)}$.
Now it suffices to check that
\[ X = \begin{pmatrix}
\delta & a_2 & \cdots & a_n \\
0 & & & \\
\vdots & & \tilde X & \\
0 & & & \\
\end{pmatrix} \]
is a representation of $\psi$: Writing $X=(x_1\cdots x_n)$, we get
$|x_1 x_{i_2} \cdots x_{i_k}|=p_{1,i_2,\ldots,i_k}$ by \eqref{eq:Xt}
for all $1<i_2<\ldots<i_k\le n$ and
$|x_{j_1} \cdots x_{j_k}|=p_{j_1,\ldots,j_k}$ by Equation \eqref{eq:ai}
for all $1<j_1<\ldots<j_k\le n$.
\end{proof}

\section{Representations with volume one vectors}

Our goal in Section \ref{sec:slk} is to characterize specializations which are restrictions of $\SL_k$-friezes to subsets of $\{1,\ldots,n\}$; note that $\SL_k$-friezes have representations in which all vectors have volume $1$.
Therefore the first step is to understand under which circumstances a representation exists such that all vectors have coprime entries.

\begin{examp}
Let $\psi : \AC_{2,3} \rightarrow \ZZ$, $\psi(t_{1,2})=\psi(t_{2,3})=\psi(t_{1,3})=3$.
Then the following matrices are both representations of $\psi$:
\[
\begin{pmatrix}
1 & 1 & 0 \\
0 & 3 & 3
\end{pmatrix},
\quad
\begin{pmatrix}
1 & 2 & 1 \\
0 & 3 & 3
\end{pmatrix}.
\]
However, only the second one has the property that all columns have volume $1$.
\end{examp}

\begin{propo}\label{gcd_cond}
Let $\DR$ be a principal ideal domain and
$X=(x_1\cdots x_n)\in \DR^{k \times n}$ be such that
$|x_i|=1$ for all $i=1,\ldots,n$. Then for all $1\le i \le n$, the ideals
\[ (|x_i x_{i_2} \cdots x_{i_k}| \mid 1\le i_2 < \ldots < i_k \le n)
= (|x_{i_1} \cdots x_{i_k}| \mid 1\le i_1 < \ldots < i_k \le n) \]
coincide.
\end{propo}
\begin{proof}
The inclusion ``$\subseteq$'' is trivial. For the converse,
without loss of generality, after a base change over $\DR$ we may assume that
$x_i$ is the first standard basis vector (remember that $|x_i|=1$). Denote $w_j$ the vector $x_j$ in which we have removed the first coordinate. Then
\[ |x_i x_{i_2} \cdots x_{i_k}| = |w_{i_2} \cdots w_{i_k}|. \]
Hence
\[ |x_{i_1} \cdots x_{i_k}| =
\sum_{\ell=1}^k (-1)^\ell X_{1,i_\ell} |w_{i_1} \cdots \widehat{w_{i_\ell}} \cdots w_{i_k}| \]
is a linear combination of $|x_i x_{i_2} \cdots x_{i_k}|$, $i_2<\ldots<i_k$.
\end{proof}

\begin{examp}
Let
$$\psi : \AC_{2,3} \rightarrow \ZZ,\quad \psi(t_{1,2})=\psi(t_{2,3})=\psi(t_{1,3})=2.$$
All ideals considered in Proposition \ref{gcd_cond} are then equal to $(2)$, so one could hope for a
representation of $\psi$ in which the volumes of the columns are all equal to $1$.
In this case, we could choose the first column to be the first standard basis vector.
But then, this representation would be the matrix
\[
\begin{pmatrix}
1 & a & b \\
0 & 2 & 2
\end{pmatrix}
\]
for some $a,b\in\ZZ$ with $a-b=1$.
Since $a$ and $b$ must have different parity, one of them is even, the corresponding column has volume $2$, and we get a contradiction.
Thus it is not possible to find a representation such that the volumes of the columns are all equal to $1$.

We will see in the theorem below that in this example, the explanation is that the projective line over $\FF_2$ has only $3$ points.
\end{examp}

\begin{lemma}\label{lem:prow}
Let $\DR$ be a principal ideal domain and
$\psi$ be a specialization of $\AC_{k,n}$ such that
\[ \psi(t_{i_1,\ldots,i_k})\in \DR\setminus\{0\} \]
for all $1\le i_1<\ldots<i_k \le n$.
Assume that there is a prime $\pp\in \DR$ with
$$ \pp \in (p_{i_1,\ldots,i_k} \mid 1\le i_1 < \ldots < i_k \le n).$$
Then if $\psi$ has a representation
$X=(x_1 \cdots x_n)\in \DR^{k\times n}$
such that $|x_i|=1$ for all $1\le i \le n$,
then there is a representation
$X'=(x'_1 \cdots x'_n)\in \DR^{k\times n}$
such that $|x'_i|=1$ and $X'_{1,i}\in(\pp)$ for all $1\le i \le n$,
\end{lemma}
\begin{proof}
Let $F:=R/(\pp)$ and $Y\in F^{k\times n}$ be the matrix
\[ Y_{i,j} := X_{i,j} + (\pp), \quad 1\le i \le k, \:\: 1\le j \le n. \]
Note that each column of $Y$ contains a nonzero entry since $|x_i|=1$ for all $i$,
and $\ell:=\rank(Y)>0$.
Moreover $\ell<k$ because $\pp$ divides each $k\times k$-minor of $X$.

Choose $\ell$ linearly independent rows $j_1,\ldots,j_\ell$ of $Y$.
Then adding $R$-multiples of the rows $X_{j_1},\ldots,X_{j_\ell}$ to the remaining rows we obtain a matrix
$X'$ in which all the entries on the $k-\ell$ other rows are divisible by $\pp$; after permuting the rows and possibly multiplying with a unit we may assume that the first row consists of elements divisible by $\pp$.
Note that $|x'_i|=1$ for all $i$ because all the applied operations correspond to multiplications with matrices of determinant $1$.
\end{proof}

\begin{defin}
For a prime $\pp$ in a principal ideal domain $R$ we write $\nu_\pp(a)$ for the $\pp$-\df{valuation} of $a\in R$, i.e.\ the maximal $m\in\ZZ_{\ge 0}$ such that $a\in (\pp^m)$.
\end{defin}

\begin{theor}\label{thm:main}
Let $\DR$ be a principal ideal domain and
$\psi$ be a specialization of $\AC_{k,n}$ such that
\[ \psi(t_{i_1,\ldots,i_k})\in \DR\setminus\{0\} \]
for all $1\le i_1<\ldots<i_k \le n$.
We write $p_{i_1,\ldots,i_k}:=\psi(t_{i_1,\ldots,i_k})$, $1\le i_1<\ldots<i_k \le n$.
Let $\varepsilon\in R$ be such that
$$(\varepsilon)=(p_{i_1,\ldots,i_k} \mid 1\le i_1 < \ldots < i_k \le n).$$
We write $(\varepsilon)=(\pp_1^{d_1})\cap\cdots\cap(\pp_e^{d_e})$ for primes $\pp_1,\ldots,\pp_e$.

Then $\psi$ has a representation
$X=(x_1 \cdots x_n)\in \DR^{k\times n}$
such that $|x_i|=1$ for all $1\le i \le n$ if and only if the following conditions are satisfied:
\begin{enumerate}
\item For all $1\le i \le n$, the ideals
\[ (p_{i,i_2,\ldots,i_k} \mid 1\le i_2 < \ldots < i_k \le n)
= (\varepsilon) \]
coincide.
\item For each $\ell=1,\ldots,e$ and $S\subseteq \{1,\ldots,n\}$,
\[ |S|=|\PP (R/(\pp_\ell))^k| \quad \Longrightarrow \quad
\exists i,j \in S \:\:
\forall i_1,\ldots,i_{k-2} \:\::\:\: \nu_{\pp_\ell}(p_{i_1,\ldots,i_{k-2},i,j})\ne \nu_{\pp_\ell}(\varepsilon). \]
\end{enumerate}
(The second condition is void for those $\pp_\ell$ for which $R/(\pp_\ell)$ is infinite.)
\end{theor}

\begin{proof}
Let $\delta_i\in R$
be such that
\[ (\delta_i)=(p_{i,i_2,\ldots,i_k} \mid 1\le i_2 < \ldots < i_k \le n). \]
Since $\delta_1\in (\varepsilon)$,
the representation $X$ constructed in (1) may be used to define a matrix
\[ X' = (x'_1\cdots x'_n) = \begin{pmatrix}
\frac{\delta_1}{\varepsilon} & \frac{a_2}{\varepsilon} & \cdots & \frac{a_n}{\varepsilon} \\
0 & & & \\
\vdots & & \tilde X & \\
0 & & & \\
\end{pmatrix} \in R^{k\times n} \]
such that $|x'_{i_1}\cdots x'_{i_k}| = p_{i_1,\ldots,i_k}/\varepsilon$ for all $i_1<\ldots<i_k$.

We first show that $\psi$ has a representation $X$ with $|x_i|=1$ for all $i$ if the conditions (1) and (2) hold.
Then
by assumption, $(\delta_i)=(\varepsilon)$, so each column vector of $X'$ has volume $1$.
Indeed, $|x'_i|$ divides each $p_{i,i_2,\ldots,i_k}/\varepsilon$ and thus also $\delta_i/\varepsilon=1$.
Moreover, starting with the $k$-th column, the entry in the last row is nonzero since $p_{1,2,\ldots,k-1,i}\ne 0$ for all $i>k-1$.

Let $\pp:=\pp_\ell \in R$ be one of the primes in the decomposition of $\varepsilon$, $F:=R/(\pp)$,
and $Y\in F^{k\times n}$ be the matrix % $X'$ after projecting each entry to $F$,
\[ Y_{i,j} := X'_{i,j} + (\pp), \quad 1\le i \le k, \:\: 1\le j \le n. \]
We count the number of different subspaces $\langle y_i\rangle\le F^k$, where $y_1,\ldots,y_n$ are the columns of $Y$.
Note that since $Y$ has rank $k$,
$\langle y_i\rangle\ne \langle y_j\rangle$ for $1\le i<j\le n$ if and only if
there exist $i_1,\ldots,i_{k-2}$ with $p_{i_1,\ldots,i_{k-2},i,j}/\varepsilon\notin (\pp)$.
Equivalently, $\langle y_i\rangle\ne \langle y_j\rangle$ if and only if
$\nu_\pp(p_{i_1,\ldots,i_{k-2},i,j})= \nu_\pp(\varepsilon)$
for some $i_1,\ldots,i_{k-2}$.
We see that condition (2) is satisfied if and only if
$F^k\setminus \cup_{i=1}^n \langle y_i\rangle\ne \emptyset$.

If $F^k\setminus \cup_{i=1}^n \langle y_i\rangle\ne \emptyset$, then we may choose
$0\ne v\in F^k\setminus \cup_{i=1}^n \langle y_i\rangle$.
Write $v=(\mu_1,\ldots,\mu_k)$ and let $j$ be minimal with $\mu_j\ne 0$.
Possibly replacing $v$ by $v/\mu_j$ we may assume that $\mu_j=1$.
Let
\[ \tilde Y_{m,i} := Y_{m,i}-\mu_m Y_{j,i} = \begin{cases}
Y_{m,i} & m\le j, \\
Y_{m,i}-\mu_m Y_{j,i} & m> j.
\end{cases} \]
In $\tilde Y$, each column contains at least one nonzero entry in a row other than the $j$-th since
otherwise, in a column $i$ we would have $Y_{m,i}=0$ for $m<j$ and $Y_{m,i}=\mu_m Y_{j,i}$ for $m>j$ or
equivalently $y_i = Y_{j,i}\cdot v$ which is excluded by the choice of $v$.

Choose $\mu^{(\ell)}_{r,s}\in R$ such that
$\mu^{(\ell)}_{m,j}+(\pp) =\mu_m$ for $m>j$ and $\mu^{(\ell)}_{r,s}=0$ for $r\le j$ or $s\ne j$;
moreover, write $j_\ell:=j$.
Then
$$X''_{m,i}:=X'_{m,i}-\sum_s \mu^{(\ell)}_{m,s} X'_{s,i}$$
defines a matrix $X''$ in which each column contains at least one entry not divisible by $\pp_\ell$ in a row which is not the $j_\ell$-th one (for the same reason as before).

The chinese remainder theorem yields $\mu_{r,s}\in R$ with
\[ \mu_{r,s} + (\pp_\ell) = \mu^{(\ell)}_{r,s} + (\pp_\ell) \]
for all $\ell=1,\ldots,e$.
We now define
$$ X'''_{m,i}:=X'_{m,i}-\sum_s \mu_{m,s} X'_{s,i}. $$
Finally we obtain a matrix $Z$ from $X'''$ by multiplying each $j$-th row
by $\pp_\ell$ for all $\ell$ with $j=j_\ell$.
This matrix $Z$ is a representation of $\psi$ and the columns of $Z$ have volume $1$
because for each prime $\pp_\ell$ and each column $i$, there is at least one entry in this column not divisible by $\pp_\ell$.

For the converse, assume that $\psi$ has a representation $X$ with $|x_i|=1$ for all $i$.
In this case Proposition \ref{gcd_cond} tells us that (1) is satisfied.
For (2), let $\pp$ be one of the primes in the decomposition of $\varepsilon$.
By Lemma \ref{lem:prow}, up to replacing $X$ by a matrix with the same properties,
without loss of generality the first row is divisible by $\pp^{\nu_\pp(\varepsilon)}$.
Thus we may again consider the matrices $X'$ and $Y$ as before.
By the previous discussion, it suffices to show
$F^k\setminus \cup_{i=1}^n \langle y_i\rangle\ne \emptyset$.
But this holds by construction of $X$ (Lemma \ref{lem:prow}) since $\langle e_1\rangle\ne \langle y_i\rangle$ for all $i$ if
$e_1$ is the first standard basis vector.
\end{proof}

We close this section with the observation that $n$ is bounded by $k$ and the prime $\pp$ when all
$\pp$-valuations are equal.

\begin{propo}\label{n2P}
Let $\DR$ be a principal ideal domain,
$\pp\in R$ be a prime,
and $0\ne p_{i_1,\ldots,i_k}\in R$, $1\le i_1<\cdots<i_k\le n$ be a specialization of $\AC_{k,n}$.
Assume that $\nu_\pp(p_{i_1,\ldots,i_k})=m$ for all $1\le i_1<\ldots<i_k\le n$.
Then
$$\binom{n}{k-1}\le |\PP (R/(\pp))^k|.$$
\end{propo}

\begin{proof}
Assume that $\nu_\pp(p_{i_1,\ldots,i_k})=m$ for all $1\le i_1<\ldots<i_k\le n$.
Let
$$\tilde p_{i_1,\ldots,i_k}:=\frac{p_{i_1,\ldots,i_k}}{\pp^m}\in R \quad \text{for}\quad 1\le i_1<\ldots<i_k\le n. $$
Note that since the Pl\"ucker relations are homogeneous polynomials, the $\tilde p_{i_1,\ldots,i_k}$ are a specialization as well.
Let $X$ be a representation of this specialization, for example the one constructed in the proof of Theorem \ref{thm:realization}.

With $F:=R/(\pp)$,
let $Y=(y_1\cdots y_n)\in F^{k\times n}$ be the matrix with entries $Y_{i,j}:=X_{i,j}+(\pp)$.
By assumption, $|y_{i_1}\cdots y_{i_k}|\ne 0$ for all $1\le i_1<\ldots<i_k\le n$ since $\nu_\pp(\tilde p_{i_1,\ldots,i_k})=0$.

Assume that
$\langle y_{i_1},\ldots, y_{i_{k-1}}\rangle =
\langle y_{j_1},\ldots, y_{j_{k-1}}\rangle\le F^k$
for $1\le i_1<\ldots<i_{k-1}\le n$ and $1\le j_1<\ldots<j_{k-1}\le n$.
Then
$y_{j_\ell}\in \langle y_{i_1},\ldots, y_{i_{k-1}}\rangle$ for all $\ell$ which is
only possible if $\{y_{j_1},\ldots, y_{j_{k-1}} \}=\{y_{i_1},\ldots, y_{i_{k-1}}\}$.
Thus we obtain $\binom{n}{k-1}$ different hyperplanes in $F^k$, hence
$\binom{n}{k-1}\le |\PP F^k|$.
\end{proof}

\section{Cluster structure}\label{sec:clu}

The algebra $\AC_{k,n}$ has the structure of a cluster algebra \cite[Thm.\ 3]{MR2205721}. For most values of $k,n$, there are infinitely many cluster variables. However, at least the elements
$t_{i_1,\ldots,i_k}$ for $1\le i_1<\ldots<i_k\le n$
are cluster variables. Those with labels $[i,i+1,\ldots,i+k-1]$ \footnote{Recall that we view the labels as elements in $\ZZ/n\ZZ$ and that $[\cdots]$ is the ordered sequence.} are usually chosen to be the frozen variables.

Given an initial cluster, all other cluster variables are rational functions of these initial cluster variables \cite{MR1887642}.
Because of the Laurent phenomenon \cite{MR1888840}, the denominators of these functions are monomials.
Moreover, if $\KR$ is an ordered field and if we specialize a cluster to positive values in $\KR$, then the values of all cluster variables will be positive \cite{MR3374957}.

\begin{examp}\label{ex:slk}
If $\DR=\ZZ$, $\KR=\QQ$, then specializing all variables in a cluster to $1$ implies that all cluster variables are positive integers.
In this case we obtain an $\SL_k$-frieze pattern:
\[
\begin{array}{ccccccccccccccccccccc}
& & & \ddots \\
1 & 0 & \cdots & 0 & 1 & p_{1,\ldots,k-1,k+1} & p_{1,\ldots,k-1,k+2} & \cdots \\
 & 1 & 0 & \cdots & 0 & 1 & p_{2,\ldots,k,k+2} & p_{2,\ldots,k,k+3} & \cdots \\
& & 1 & 0 & \cdots & 0 & 1 & p_{3,\ldots,k+1,k+3} & p_{3,\ldots,k+1,k+4} & \cdots \\
& & & & & & & \ddots
\end{array}
\]
\end{examp}

There is a distinguished set of clusters which corresponds to maximal non-crossing collections of $k$-subpolygons in the regular $n$-gon:

\begin{defin}[{\cite[Def.\ 3]{MR2205721}}]
Two subsets $I$ and $J$ in $\{1,\ldots,n\}$ are said to be \df{non-crossing} if no chord in
the regular $n$-gon having end points labelled by elements from $I\setminus J$ crosses any chord with
end points labelled by elements from $J\setminus I$.
\end{defin}

A maximal collection of non-crossing $k$-subpolygons $I_1,\ldots,I_{k(n-k)+1}$ in an $n$-gon
corresponds to the cluster $\{ t_{I_1},\ldots,t_{I_{k(n-k)+1}}\}$.
Such a cluster may be visualized by a Postnikov arrangement which in turn can be used to obtain a quiver.
This quiver then produces all (usually infinitely many) clusters after quiver mutations.
For example, only $\AC_{3,6}$, $\AC_{3,7}$, and $\AC_{3,8}$ are of finite type if $2<k\le \frac{n}{2}$ (\cite[Thm.\ 5]{MR2205721}).

\begin{examp}
If $k=3$ and $n=7$, the following collection is non-crossing:
\[
\{1,2,3\},\{1,2,5\},\{1,2,7\},\{1,5,7\},\{1,6,7\},\{2,3,4\},\{2,3,5\},
\]
\[
\{2,4,5\},\{2,5,6\},\{2,5,7\},\{3,4,5\},\{4,5,6\},\{5,6,7\}
\]
Specializing this cluster to $1$ gives the $\SL_3$-frieze pattern (see Example \ref{ex:slk}):
\[
\begin{array}{rrrrrrr}
0 & 0 & 1 & 2 & 1 & 3 & 1\\
1 & 0 & 0 & 1 & 1 & 4 & 2\\
5 & 1 & 0 & 0 & 1 & 6 & 5\\
3 & 1 & 1 & 0 & 0 & 1 & 2\\
2 & 1 & 2 & 1 & 0 & 0 & 1\\
1 & 2 & 8 & 7 & 1 & 0 & 0\\
0 & 1 & 5 & 5 & 1 & 1 & 0
\end{array}
\]
\end{examp}

\section{Crystallographic arrangements from clusters}\label{sec:cry}

If we specialize all variables of a cluster to $\pm 1$ instead of $1$, by the Laurent phenomenon we still
obtain a specialization to the ring of integers since all denominators of cluster variables are specialized to $\pm 1$. However, some of the variables will possibly be specialized to $0$.
A representation of this specialization thus also represents the matroid defined by these linear dependencies.

Depending on the choice of a cluster and of values $\pm 1$, the resulting matroids can have
very special properties.
For example, we obtain the reflection arrangements of type $A_3$ and $B_3$:

\begin{examp}
Let $k=3$, $n=6$ and choose the cluster consisting of all triangles $[i,i+1,i+2]$ (frozen variables) and of the following non-crossing sets:
\[ \{1,2,4\}, \{1,2,5\}, \{1,3,4\}, \{1,4,5\}  \]
Specializing all variables in this cluster to $1$ except $\{1,4,5\}$ which we map to $-1$, we obtain as a representation:
\[
\begin{pmatrix}
0 & 1 & 1 & 0 & 1 & 0\\
0 & 0 & 1 & 1 & 1 & 1\\
1 & 0 & 0 & 1 & 1 & 0
\end{pmatrix}.
\]
The columns are the positive roots of the root system of type $A_3$.
This may also be viewed as a sort of $\SL_3$-frieze pattern in which the entries are not necessarily positive.
\end{examp}

\begin{examp}
Let $k=3$, $n=9$ and choose the cluster consisting of all the frozen variables and of
\[
\{1,2,4\},
\{1,2,8\},
\{1,4,8\},
\{2,4,5\},
\{2,4,8\},
\{4,5,7\},
\{4,5,8\},
\{4,7,8\},
\{4,8,9\},
\{5,7,8\}.
\]
We specialize the frozen variables to $1$ and the other variables of the cluster to
\[ 1, 1, -1, 1, 1, -1, 1, 1, -1, -1. \]
Then a representation is
\[
\begin{pmatrix}
0 & 1 & 1 & 0 & 0 & 1 & -1 & 1 & 0\\
0 & 0 & 1 & 1 & -1 & 2 & -1 & 1 & 1\\
1 & 0 & 0 & 1 & 0 & 2 & -1 & 2 & 2
\end{pmatrix}
\]
which consists of $\alpha$ or $-\alpha$ for each positive root $\alpha$ of the root system of type $B_3$.
\end{examp}

It is not obvious what cluster and which values are required to obtain reflection arrangements in general.

\begin{question}
What is the structure of a cluster that produces a Weyl arrangement when specializing the variables to the right values $\pm 1$?
\end{question}

An explanation for the fact why some Weyl arrangements appear as specializations is that the volume of their chambers is $1$. Thus those $k$-gons which correspond to chambers are good candidates for frozen variables.
In particular, it is reasonable to look for clusters for which all frozen variables correspond to chambers.

\begin{examp}
For $k=3$, $n=16$, the following sets are non-crossing:
\begin{eqnarray*}
\{1,3,4\},\{1,3,14\},\{1,3,16\},\{1,14,16\},\{3,4,6\},\{3,4,8\},\{3,4,14\},\{3,14,15\}, \\
\{3,14,16\},\{4,6,7\},\{4,6,8\},\{4,7,8\},\{4,8,9\},\{4,8,10\},\{4,8,12\},\{4,8,14\}, \\
\{4,10,12\},\{4,12,13\},\{4,12,14\},\{4,13,14\},\{4,14,15\},\{8,10,12\},\{9,10,12\},\{10,12,13\}
\end{eqnarray*}
If we specialize these to
\[
1, 1, 1, 1, 1, -1, 1, -1, 1, -1, 1, 1, 1, -1, 1, -1, -1, 1, -1, -1, 1, 1, -1, -1
\]
and the frozen variables to $1$, then we obtain
\[
\left(
\begin{array}{rrrrrrrrrrrrrrrr}
0 & 1 & 1 & 0 & 0 & 1 & 1 & 0 & -1 & 1 & 1 & -1 & -2 & 1 & 1 & 0\\
0 & 0 & 1 & 1 & -1 & 1 & 1 & -1 & -1 & 2 & 3 & -2 & -3 & 2 & 2 & 1\\
1 & 0 & 0 & 1 & 0 & 1 & 2 & -2 & -3 & 4 & 4 & -3 & -4 & 2 & 1 & 3
\end{array}
\right)
\]
as a representation.
The columns define a simplicial arrangement denoted $\AC(16,3)$ by Gr\"unbaum \cite{MR2485643}.
It is also a crystallographic arrangement corresponding to the $13$-th finite Weyl groupoid of rank three in the numbering of \cite{MR3341467}.
\end{examp}

\section{Subpolygons in $\SL_k$-friezes}\label{sec:slk}

In this section, we obtain a condition for subpolygons in $\SL_k$-friezes in general.
Using Theorem \ref{thm:main}, it is easy to recover the previous results \cite[Thm.\ 3.2]{MR4311087} and \cite[Thm.\ 3.2]{MR4311087} from the introduction ($k=2$), since the Euclidean algorithm provides the missing vectors.

\begin{figure}
\begingroup%
  \makeatletter%
  \providecommand\rotatebox[2]{#2}%
  \newcommand*\fsize{\dimexpr\f@size pt\relax}%
  \newcommand*\lineheight[1]{\fontsize{\fsize}{#1\fsize}\selectfont}%
  \ifx\svgwidth\undefined%
    \setlength{\unitlength}{270bp}%
    \ifx\svgscale\undefined%
      \relax%
    \else%
      \setlength{\unitlength}{\unitlength * \real{\svgscale}}%
    \fi%
  \else%
    \setlength{\unitlength}{\svgwidth}%
  \fi%
  \global\let\svgwidth\undefined%
  \global\let\svgscale\undefined%
  \makeatother%
  \begin{picture}(1,0.67734383)%
    \lineheight{1}%
    \setlength\tabcolsep{0pt}%
    \put(0,0){\includegraphics[width=\unitlength,page=1]{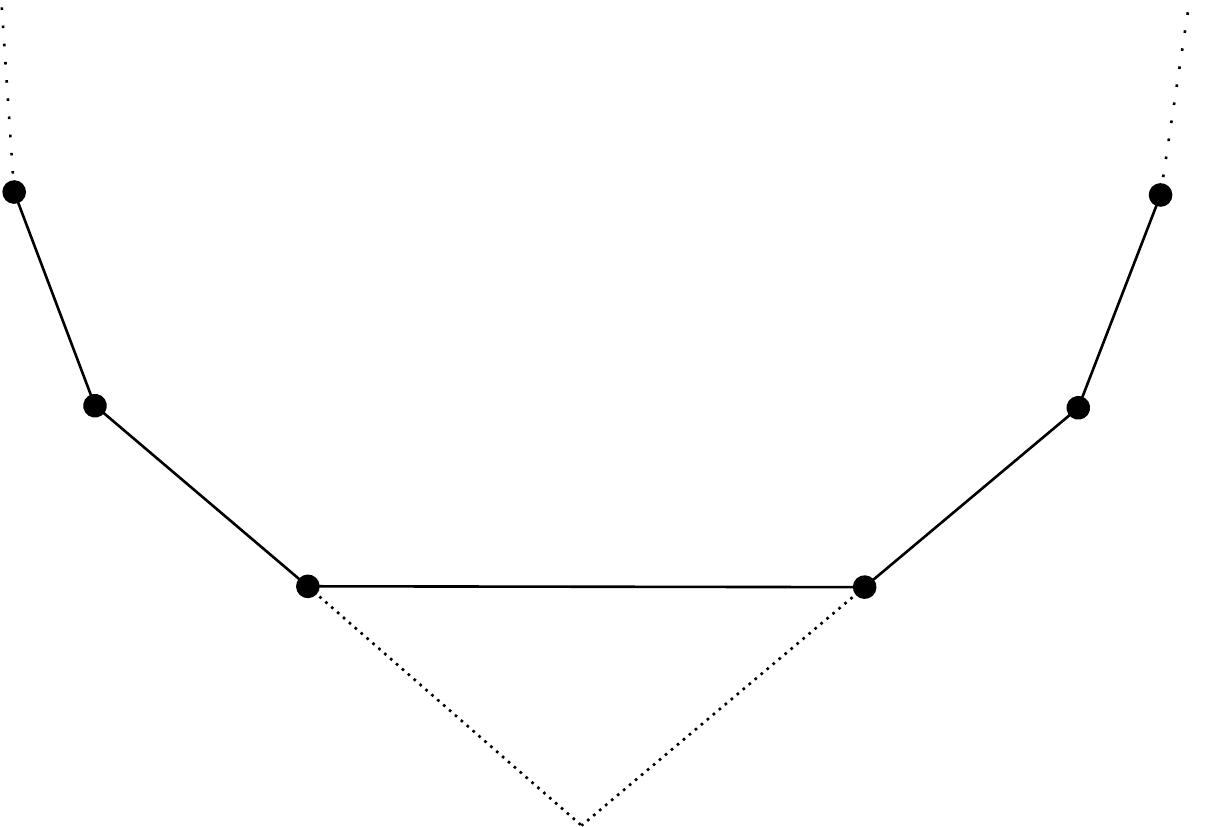}}%
    \put(0.01833648,0.30176774){\color[rgb]{0,0,0}\makebox(0,0)[lt]{\lineheight{1.25}\smash{\begin{tabular}[t]{l}$x_1$\end{tabular}}}}%
    \put(0.19660924,0.1527761){\color[rgb]{0,0,0}\makebox(0,0)[lt]{\lineheight{1.25}\smash{\begin{tabular}[t]{l}$x_2$\end{tabular}}}}%
    \put(0.72707451,0.16291272){\color[rgb]{0,0,0}\makebox(0,0)[lt]{\lineheight{1.25}\smash{\begin{tabular}[t]{l}$x_3$\end{tabular}}}}%
    \put(0.90035256,0.30920126){\color[rgb]{0,0,0}\makebox(0,0)[lt]{\lineheight{1.25}\smash{\begin{tabular}[t]{l}$x_4$\end{tabular}}}}%
    \put(0.31939667,0.07244187){\color[rgb]{0,0,0}\makebox(0,0)[lt]{\lineheight{1.25}\smash{\begin{tabular}[t]{l}$\alpha_1$\end{tabular}}}}%
    \put(0.46058466,0.2209843){\color[rgb]{0,0,0}\makebox(0,0)[lt]{\lineheight{1.25}\smash{\begin{tabular}[t]{l}$\alpha_2$\end{tabular}}}}%
    \put(0.59196844,0.06647312){\color[rgb]{0,0,0}\makebox(0,0)[lt]{\lineheight{1.25}\smash{\begin{tabular}[t]{l}$\alpha_3$\end{tabular}}}}%
    \put(0.49050009,0.13411864){\color[rgb]{0,0,0}\makebox(0,0)[lt]{\lineheight{1.25}\smash{\begin{tabular}[t]{l}$y$\end{tabular}}}}%
    \put(0,0){\includegraphics[width=\unitlength,page=2]{cone.pdf}}%
  \end{picture}%
\endgroup%
\caption{Cone in the proof of Theorem \ref{thm:slk} for $k=3$ (projective plane).\label{fig:cone}}
\end{figure}

The general case (arbitrary $k$) is much more difficult. The idea is to inductively include elements to a given sequence by considering a simplicial cone of volume greater than $1$.
One has to choose these new elements close to the border in order to decrease the values of the frozen variables:

\begin{theor}\label{thm:slk}
Let $x_1,\ldots,x_n\in \ZZ^k$ be with $|x_i|=1$ for all $i=1,\ldots,n$.
Assume that $|x_{i_1}\ldots x_{i_k}|> 0$ for all $1\le i_1<\ldots <i_k\le n$.
Then there exist $\yy_1,\ldots,\yy_\ell\in\ZZ^k$, $\ell\ge n$ and $i_1<\ldots<i_n$ with
\[ \yy_{i_r}=x_r,
\quad |\yy_{i_1}\ldots \yy_{i_k}|>0,
\quad \text{and} \quad
|\yy_{[i,i+1,\ldots,i+k-1]}|=1 \]
for all $r=1,\ldots,n$, $i_1<\ldots<i_k$, and $i=1,\ldots,\ell$, where
$\yy_{[i,i+1,\ldots,i+k-1]}$ denotes the matrix with columns $\yy_{i},\ldots,\yy_{i+k-1}$
sorted by the reduced indices.\\
In other words, the ``frieze with coefficients'' may be extended to an $\SL_k$-frieze pattern.
\end{theor}
\begin{proof}
Consider without loss of generality the first $2k-2$ elements $x_1,\ldots,x_{2k-2}$ (partly repeating if $2k-2>n$).
We want to understand what happens when a $y\in \ZZ^k$ is included between $x_{k-1}$ and $x_k$,
i.e.\ when replacing the sequence with
\[ x_1,\ldots,x_{k-1},y,x_k,\ldots,x_{2k-2}. \]
For $j=1,\ldots,k$ let $\alpha_j\in\ZZ^k$ be such that the $i$-th entry of $\alpha_j$ is the
$(k-1)\times(k-1)$ minor
of the $k\times(k-1)$-matrix $(x_j\cdots x_{j+k-2})$ to the rows $1,\ldots,\hat i,\ldots,k$.
Then after possibly multiplying $\alpha_j$ by $-1$,
\[ \langle x_j,\ldots,x_{j+k-1}\rangle = \alpha_j^\perp \quad \text{and} \quad
\left( \:(\alpha_j,y)>0 \quad \Longleftrightarrow \quad
|x_j \ldots x_{k-1}\: y\: x_k \ldots x_{j+k-2}|>0\right) \]
where $\perp, (\cdot,\cdot)$ denote the standard scalar product.
Thus the hyperplanes $\alpha_j^\perp$ are the walls of a simplicial cone containing solutions $y$ to our extension problem (see Figure \ref{fig:cone} for $k=3$).

Let $A$ be the matrix with rows $\alpha_1,\ldots,\alpha_k$.
As before, we write $p_{i_1,\ldots,i_k}:=|x_{i_1}\cdots x_{i_k}|$.
Then we see that
\begin{equation}\label{eq:acoord}
A \cdot (x_1\cdots x_k) =
\begin{pmatrix}
0 & \cdots & \cdots & 0 & p_{1,\ldots,k} \\
\pm p_{1,\ldots,k} & 0 & \cdots& \cdots  & 0 \\
\pm p_{1,3,\ldots,k+1} & \pm p_{2,\ldots,k+1} & 0 & \cdots \\
\vdots & & \ddots & \ddots \\
\pm p_{1,k,\ldots,2k-2} & \cdots & \cdots & \pm p_{k-1,\ldots,2k-2} & 0
\end{pmatrix}
\end{equation}
and hence
\[ \det(A) = \pm \prod_{i=1}^{k-1} p_{i,\ldots,i+k-1}.\]
We obtain a parallelotope whose volume is the product of the first $k$ frozen variables.

We prove the claim by induction over $D:=\prod_{i=1}^n p_{i,\ldots,i+k-1}$.
If $D=1$, then there is no need to include a $y$ since the frozen variables $p_{i,\ldots,i+k-1}$, $i=1,\ldots,n$ are already equal to $1$.
Assume that $D>1$, without loss of generality $\det(A)\ne \pm 1$.

We first prove the existence of $\yy_1,\ldots,\yy_\ell$ extending the sequence of $x_i$
such that
$$|\yy_{[i,i+1, \ldots,i+k-2]}|=1$$
for all $i=1,\ldots,\ell$.
Assume without loss of generality that some of the $|\alpha_i|=p_{i,\ldots,i+k-2}$ are greater than $1$ for $1\le i\le k$.
Let $\tilde A$ be the matrix with rows $\alpha_j/|\alpha_j|$.
After a base change on $\ZZ^k$, we may assume that $\tilde A$ is in Hermite normal form (on the columns), i.e.\ $\tilde A=(a_{i,j})\in\ZZ^{k\times k}$ with
\[ a_{i,j}=0 \:\:\text{for}\:\: i<j, \quad
\text{and}\quad 0\le a_{i,j}<a_{i,i} \:\:\text{for}\:\: i>j. \]
Since $\tilde A$ is upper-triangular, we find a $y\in\ZZ^k$ such that
\[ \tilde A \cdot y = z = (z_i)_i, \quad
0 < z_i \le a_{i,i} \:\: i=1,\ldots,k. \]
In particular,
$|x_j \ldots x_{k-1}\: y\: x_k \ldots x_{j+k-2}| = z_j \le a_{j,j}$
and thus
\[
\prod_{j=1}^k |x_j \ldots x_{k-1}\: y\: x_k \ldots x_{j+k-2}|
\le \det(\tilde A) < \det(A)
= \prod_{j=1}^{k-1} |x_j \ldots x_{j+k-1}|.
\]
Including $y$, the new sequence satisfies our assumptions:
the new frozen variables are positive since $z_j>0$ for all $j$;
the remaining variables are all positive because the new frozen variables extend
a previous cluster in which all variables were positive.
The product $D$ has decreased and we are finished by induction.

We may now assume without loss of generality that $|\alpha_i|=1$ for all $i=1,\ldots,k$, and moreover that $p_{1,\ldots,k}>1$. Since $\det(A)=\det(\tilde A)$ in this case, we have to be more careful with the choice of $y$.
Again, we choose a basis such that $A=\tilde A$ is in Hermite normal form.
We compute the coordinates of $x_1,\ldots,x_k$ with respect to this basis.
Recall (Equation \ref{eq:acoord}) that for $1\le i<k$,
\[ \alpha_1(x_i)=\cdots=\alpha_i(x_i)=0, \quad
\alpha_{i+1}(x_i)=(-1)^i p_{i,\ldots,i+k-1}. \]
Using $\alpha_1(x_k)=p_{1,\ldots,k}$, $\alpha_k(x_k)=0$,
the fact that $A$ is upper triangular, and $a_{1,1}=1$ (because $|\alpha_1|=1$), we get
\[
(x_1\cdots x_k) =
\begin{pmatrix}
0 & \cdots & \cdots & 0 & p_{1,\ldots,k} \\
\frac{p_{1,\ldots,k}}{a_{2,2}} & 0 & \cdots& \vdots  & * \\
* & -\frac{p_{2,\ldots,k+1}}{a_{3,3}} & \ddots & \vdots & \vdots \\
\vdots & \ddots & \ddots & 0 & *\\
* & \cdots & * & (-1)^k\frac{p_{k-1,\ldots,2k-2}}{a_{k,k}} & 0
\end{pmatrix},
\]
and this should be a matrix in $\ZZ^{k\times k}$.
So $p_{i,\ldots,i+k-1}$ is divisible by $a_{i+1,i+1}$ for all $i$.
On the other hand,
\[
\prod_{i=2}^k a_{i,i} = \det(A) = \pm \prod_{i=1}^{k-1} p_{i,\ldots,i+k-1},
\]
which is only possible if $a_{i+1,i+1}=p_{i,\ldots,i+k-1}$ for all $i$.
Now $0\le a_{2,1} < a_{2,2} = p_{1,\ldots,k}$ because $A$ is in Hermite normal form.
Since $p_{1,\ldots,k}>1$ and $|\alpha_2|=1$, $a_{2,1}=0$ is excluded.
As before we look for a $y\in\ZZ^k$ such that $A\cdot y=z$ has positive coordinates;
the form of $A$ allows a $y$ such that $0<z_i\le a_{i,i}$ for all $i$.
In particular, $a_{1,1}=1$, so we can choose $y_1=1$ and get $z_1=1$.
For $y_2$ we get the equation $a_{2,1} y_1 + a_{2,2} y_2 = z_2$, or
\[ a_{2,1} + y_2\cdot p_{1,\ldots,k} = z_2. \]
With $y_2=0$ we get $z_2=a_{2,1}<p_{1,\ldots,k}$.
Hence after choosing the remaining $y_i$,
\[ \prod_{j=1}^k |x_j \ldots x_{k-1}\: y\: x_k \ldots x_{j+k-2}|=\prod_{j=1}^k z_j < \prod_{j=1}^k a_{j,j} = \det(A). \]
Including $y$ in the sequence decreases the product $D$ of all $p_{i,\ldots,i+k-1}$;
moreover, the new sequence satisfies the assumptions ($z_j>0$ for all $j$ and these extend a previous cluster),
so we are finished by induction.
\end{proof}

\newcommand{\etalchar}[1]{$^{#1}$}
\providecommand{\bysame}{\leavevmode\hbox to3em{\hrulefill}\thinspace}
\providecommand{\MR}{\relax\ifhmode\unskip\space\fi MR }
\providecommand{\MRhref}[2]{%
  \href{http://www.ams.org/mathscinet-getitem?mr=#1}{#2}
}
\providecommand{\href}[2]{#2}

\end{document}